\documentclass{article}
\usepackage{natbib}
\usepackage{graphicx}
\usepackage{amsfonts}
\usepackage{amsmath}
\usepackage{arxiv}
\usepackage{url}
\usepackage{amsthm}
\usepackage{amssymb}

\usepackage[utf8]{inputenc} 
\usepackage[T1]{fontenc}    
\usepackage{hyperref}       
\usepackage{url}            
\usepackage{booktabs}       
\usepackage{amsfonts}       
\usepackage{nicefrac}       
\usepackage{microtype}      
\usepackage{lipsum}
\usepackage{graphicx}
\graphicspath{ {./images/} } 

\usepackage{times}
\usepackage{hyperref}
\usepackage[utf8]{inputenc} 
\usepackage{hyperref}       
\usepackage{url}            
\usepackage{booktabs}       
\usepackage{amsfonts}       
\usepackage{nicefrac}       
\usepackage{microtype}      
\usepackage{lipsum}
\usepackage{dirtytalk}
\usepackage{mathrsfs}
\usepackage{mathtools}
\usepackage{transparent}
\usepackage{times}
\usepackage{graphicx} 
\usepackage{wrapfig}
\usepackage{relsize,exscale}
\usepackage{nameref}

\newtheorem{defn}{Definition}
\newtheorem{thm}{Theorem}
\newtheorem{prop}{Proposition}
\newtheorem{rem}{Remark}

\newtheorem{ex}{Example}
\newtheorem{hyp}{Assumption}
\newcommand{\R}{\mathbb{R}}

\newcommand{\norm}[1]{{\left\lVert #1 \right\rVert}}

\makeatletter
\def\namedlabel#1#2{\begingroup
   \def\@currentlabel{#2}%
   \label{#1}\endgroup
}
\makeatother

\DeclareMathOperator*{\argmin}{arg\,min}

\def\P{\mathbb P}
\def\E{\mathbb E}

\def\N{\mathbb N}

\def\X{\mathcal X}

\title{A Consistent Extension of Discrete Optimal Transport Maps\\ for Machine Learning Applications}

\date{\vspace{-5ex}}
\author{Lucas De Lara$^{(1)}$\footnotemark[1] \ \ and  \ Alberto Gonz\'alez-Sanz$^{(2)}$\footnotemark[1] \ \ and \ Jean-Michel Loubes $^{(3)}$\thanks{Research partially supported by the AI Interdisciplinary Institute ANITI, which is funded by the French “Investing for the Future – PIA3” program under the Grant agreement ANR-19-PI3A-0004.}\\  $\,$ \\ 
IMT, Universit\'e de Toulouse III
France\\ $\,$ \\ 
$^{(1)}$lucas.de\_lara@math.univ-toulouse.fr \quad $^{(2)}$alberto.gonzalez sanz@math.univ-toulouse.fr \\ $^{(3)}$loubes@math.univ-toulouse.fr}

\begin{document}

\maketitle

\begin{abstract}%
Optimal transport maps define a one-to-one correspondence between probability distributions, and as such have grown popular for machine learning applications. However, these maps are generally defined on empirical observations and cannot be generalized to new samples while preserving asymptotic properties. We extend a novel method to learn a consistent estimator of a continuous optimal transport map from two empirical distributions. The consequences of this work are two-fold: first, it enables to extend the transport plan to new observations without computing again the discrete optimal transport map; second, it provides statistical guarantees to machine learning applications of optimal transport. We illustrate the strength of this approach by deriving a consistent framework for transport-based counterfactual explanations in fairness.
\end{abstract}

\textbf{Keywords:} Optimal Transport, Counterfactuals, Explanability, Fairness.

\section{Introduction}

Over the last past years, Optimal Transport (OT) methods have grown popular for machine learning applications. Signal analysis \citep{kolouri2017optimal}, domain adaptation \citep{7586038}, transfer learning \citep{gayraud2017optimal} or fairness in machine learning \citep{jiang2020wasserstein,pmlr-v97-gordaliza19a} for instance have proposed new methods that make use of optimal transport maps. Given two distributions $\mu$ and $\nu$ satisfying some assumptions, such a map $T$ has the property of \emph{pushing forward} a measure to another in the sense that if a random variable $X$ follows the distribution $\mu$, then its image $T(X)$ follows the distribution $\nu$. This map comes as a tool to transform the distribution of observations.

However, since only empirical distributions are observed, the continuous optimal transport  is transformed into an empirical problem. Optimal transport for empirical distributions has been widely studied from both a theoretical and a computational point of view. We refer for instance to~\cite{peyre2019computational} and references therein. The obtained empirical maps between observations suffer some important drawbacks when implementing machine learning methods relying on OT. As they are one-to-one correspondences between the points used to compute the optimal transport, they are only defined on these observations, preventing their use for new inputs.

To cope with this issue, either the map must be recomputed for each new data set or one must use a continuous approximation extending the empirical map to observations out of the support of the empirical distribution. Previous research on the latter topic includes GAN approximations of the OT map \citep{black2020fliptest} and Monte-Carlo approximations of the dual parameters \citep{chiappa2021fairness}. However, these methods don't provide consistent estimators, in the sense that the obtained transport plans are not asymptotically close to the continuous OT map as the sample size increases.

In this paper, we propose to fill the gap between continuous and empirical transport by considering a statistically consistent interpolation of the OT map for discrete measures. On the basis of the interpolation provided in \cite{delbarrio2020centeroutward}, we generalize their results and prove that it is possible to learn from empirical observations a suitable OT map for machine learning methods. We then utilize this interpolation to derive the first consistent framework for empirically-based counterfactual explanations to audit the fairness of binary classifiers, extending the work in~\cite{black2020fliptest}.

\section{Learning an OT map for new observations}\label{learning}

Let $\mu_0$ and $\mu_1$ be two unknown probability measures on $\R^d$ whose respective supports are denoted by $\X_0$ and $\X_1$. In this section, we address the problem of learning the optimal transport map between $\mu_0$ and $\mu_1$ from data points.

\subsection{Background in Optimal Transport}

Let $\norm{\cdot}$ denote the euclidean norm associated with the scalar product $\langle \cdot,\cdot \rangle$. The \emph{optimal transport map} between $\mu_0$ and $\mu_1$ with respect to the squared euclidean cost is defined as the solution to the following \emph{Monge problem}:
\begin{equation}\label{monge}
    \min_{T:\ T_\sharp \mu_0 = \mu_1} \int_{\R^d} \norm{x-T(x)}^2 d\mu_0(x), 
\end{equation}
where $T_\sharp \mu_0 = \mu_1$ denotes that $T$ \emph{pushes forward} $\mu_0$ to $\mu_1$, namely $\mu_1(B) := \mu_0 (T^{-1}(B))$ for any measurable set $B\subset \R^d$. Suppose that $\mu_0$ is absolutely continuous with respect to the Lebesgue measure $\ell_d$ in $\R^d$, and that both $\mu_0$ and $\mu_1$  have finite second order moments. Theorem~2.12 in \cite{villani2003topics} states that there exists an unique solution to \eqref{monge} $T : \X_0 \to \R^d$ called the \emph{Brenier map}. This map coincides $\mu_0$-almost surely with the gradient of a convex function, and in consequence has a \emph{cyclically monotone} graph. Recall that set $S \subset \R^d \times \R^d$ is cyclically monotone if any finite set $\{ (x_k,y_k)\}_{k=1}^N\subset S$ satisfies 
\begin{align*}
   \sum_{k=1}^{N-1} \langle y_k,x_{k+1}-x_{k} \rangle +\langle y_N,x_1-x_N \rangle \leq 0.
\end{align*}
 Such a set is contained in the graph of the \emph{subdifferential} of a convex function, see  \citep{rockafellar1970convex}. The subdifferential at a point $x \in \R^d$ of a convex function $\psi$ is defined as the set

\begin{equation*}\label{subdiff}
    \partial \psi(x) := \{ y \in \R^d |\forall z \in \R^d, \psi(z) - \psi(x) \geq \langle y, z-x \rangle \}.
\end{equation*}
We say that a \emph{multivalued map} $F:\R^d\rightarrow 2^{\R^d}$ is cyclically monotone if its graph is.

In a practical setting, we only have access to samples from $\mu_0$ and $\mu_1$, and consequently we can't solve \eqref{monge}. However, we can compute a discrete optimal transport map between the empirical measures. Consider two $n$-samples $\{x^0_1, \dots, x^0_n\}$ and  $\{x^1_1, \dots, x^1_n\}$ respectively drawn from $\mu_0$ and $\mu_1$. They define the empirical measures 
\begin{align*}
   \mu^n_{0} :=\frac{1}{n}\sum_{k=1}^{n}\delta_{x^0_k} \text{\ \ and \ \ } \mu^n_{1} :=\frac{1}{n}\sum_{k=1}^{n}\delta_{x^1_k}.
\end{align*}

The \emph{discrete} Monge problem between $\mu^n_{0}$ and $\mu^n_{1}$ is
\begin{align}\label{discrete_ot}
  \min_{T_n\in \mathcal{T}_n}\frac{1}{n}\sum_{k=1}^{n}||{x^0_k}-T_n(x^0_k)||^2,
\end{align}
where $\mathcal{T}_n$ denotes the set of bijections from $\{x^0_i\}^n_{i=1}$ to $\{x^1_i\}^n_{i=1}$. Problem \eqref{discrete_ot} defines an unique solution $T_n$ referred as the \emph{discrete} optimal transport map between the two samples. This solution is such that $\big\{(x^0_k, T_n(x^0_k)) \big\}_{k=1}^n$ is {cyclically monotone}.

In this paper, we focus on the problem of estimating the optimal transport map $T$ solving \eqref{monge}. As mentioned in the introduction, the solution $T_n$ to $\eqref{discrete_ot}$ is not a suitable estimator because it has finite input and output spaces, whereas $T$ maps the whole domains. As a consequence, the empirical map cannot generalize to new observations. This limitation triggered the need for regularized approaches: a topic we explore next.


\subsection{Smooth Interpolation}\label{interpolation}


The heuristic approximation of a continuous OT map proposed in \cite{black2020fliptest} handles new observations and has a satisfying empirical behaviour, but is not guaranteed to converge to the true OT map as the sample size increases. The problem of constructing an approximation able to generalize to new observations while being statistically consistent crucially raises the question of which properties of continuous optimal transport must be preserved by the empirical estimator.

Recall that in one dimension the continuous optimal transport map between two probability measures is a non-decreasing function $T$ such that $T_\sharp \mu_0= \mu_1$. Then, natural extensions and regularization are made by preserving that property. In several dimensions, the cyclically monotone property substitutes the non-decreasing one. For the purpose of generalizing the notion of distribution function to higher dimensions, \cite{delbarrio2020centeroutward} designed such an extension of ${T}_n$ that converges to $T$ as the sample size increases. We briefly present the construction hereafter, and refer to \cite{delbarrio2020centeroutward} for further details.

The idea is to extend the discrete map $T_n : \{x^0_i\}^n_{i=1} \to \{x^1_i\}^n_{i=1}$ to a continuous one $\overline{T}_n : \R^d \to \R^d$ by regularizing a piece-wise constant approximation of $T_n$. The first step consists in solving \eqref{discrete_ot} and permuting the observations so that for every $i \in \{1,\ldots,n\}$, $T_n(x^0_i)=x^1_i$. Once the samples are aligned, we look for the parameters $\varepsilon_0$ and $\psi \in \R^n$ defined as the solutions to the linear program 

\begin{equation}\label{primal}
 \begin{split}
    \max_{\psi \in \R^n, \varepsilon_0 \in \R} &\varepsilon_0\\
    \text{ s.t. } \langle x^0_i, x^1_i - x^1_j \rangle \geq \psi_i &- \psi_j + 2 \varepsilon_0,\ i \neq j.
 \end{split}
\end{equation}

Recall that $\big\{(x^0_i,x^1_i)\big\}^n_{i=1}$ is cyclically monotone, and consequently is contained in the graph of the subdifferential of some convex function. Since this is a finite set, there exist several convex functions satisfying this property. For any of them denoted by $\overline{\varphi}_n$, its convex conjugate $\overline{\varphi}^*_n := \sup_{z \in \R^d} \{ \langle z, \cdot \rangle - \overline{\varphi}_n(z)\}$ is such that

$$
    \overline{\varphi}_n^*(x^1_i) - \overline{\varphi}^*_n(x^1_j) \leq \langle x^0_i,x^1_i - x^1_j \rangle.
$$

The idea behind \eqref{primal} is to find the most regular candidate convex function $\overline{\varphi}_n$ by maximizing the strict convexity of $\overline{\varphi}^*_n$. Proposition 3.1 in \cite{delbarrio2020centeroutward} implies that \eqref{primal} is feasible. In practice, we solve \eqref{primal} by applying Karp's algorithm \citep{karp1978characterization} on its dual formulation:

\begin{equation}\label{dual}
 \begin{split}
    \min_{z_{i,j} : i \neq j} &\sum_{i,j : i \neq j} z_{i,j} \langle x^0_i,x^1_i \rangle\\
    \text{ s.t. } &\sum_{j : j \neq i} (z_{i,j}-z_{j,i})=0,\\ &\sum_{i,j : i \neq j} z_{i,j}=1,\ z_{i,j} \geq 0,\ i,j=1,\ldots,n.
 \end{split}
\end{equation}

Next, define the following convex function

\begin{equation} \label{extended}
    \tilde{\varphi}_n(x) := \max_{1 \leq i \leq n} \big\{ \langle x , x^0_i \rangle - \psi_i \big\}.
\end{equation}

Note that $\nabla \tilde{\varphi}_n$, wherever it is well-defined, is a piece-wise constant interpolation of $T_n$. To obtain a regular interpolation defined everywhere and preserving the cyclical monotonicity we consider the Moreau-Yosida regularization of $\tilde{\varphi}_n$ given by

$$
    \varphi_n(x) := \inf_{z \in \R^d} \big\{ \tilde{\varphi}_n(z) + \frac{1}{2 \varepsilon_0} ||z-x||^2 \big\}.
$$

Such a regularization is differentiable everywhere. Then, the mapping from $\R^d$ to $\R^d$ defined as $\overline{T}_n := \nabla \varphi_n$ satisfies the following properties:

\begin{enumerate}
    \item $\overline{T}_n$ is continuous,
    \item $\overline{T}_n$ is cyclically monotone,
    \item for all ${i \in \{1,\ldots,n\}},\ \overline{T}_n(x^0_i) = x^1_i = T_n(x^0_i)$,
    \item for all ${x \in \R^d}$, $\overline{T}_n(x) $ belongs to the convex hull of $ \{x^1_1,\ldots,x^1_n\} $.
\end{enumerate}

A more explicit expression of $\overline{T}_n$ can be derived using the gradient formula of Moreau-Yosida regularizations. For $g : \R^d \to \R \cup \{+\infty\}$ a proper convex lower-semicontinuous function, the proximal operator of $g$ is defined on $\R^d$ by

\begin{equation}\label{prox}
    \text{prox}_g(x) := \argmin_{z \in \R^d} \big\{ g(z) + \frac{1}{2} ||z-x||^2 \big\}.
\end{equation}

Note that it is well-defined since the minimized function is strictly convex. Then, according to Theorem 2.26 in \cite{rockafellar2009variational} we have

\begin{align}\label{formula}
    \overline{T}_n(x) = \frac{1}{\varepsilon_0} \big( x-\text{prox}_{\varepsilon_0 \tilde{\varphi_n}}(x) \big).
\end{align}

The interpolation of each new input $x$ is numerically computed by solving the optimization problem $\text{prox}_{\varepsilon_0 \tilde{\varphi_n}}(x)$. As a consequence, generalizing with $\overline{T}_n$ is not computationally free as we must compute proximal operators. Let's benchmark this approach against classical discrete OT.

Suppose for instance that after constructing $T_n$ on $\{x^0_i\}^n_{i=1}$ and $\{x^1_i\}^n_{i=1}$ we must generalize the OT map on a new sample $\{\bar{x}^0_i\}^m_{i=1} \sim \mu_0$ such that $m \leq n$. Without additional observations from $\mu_1$, we are limited to: computing for each $\bar{x}^0_i$ its closest counterpart in $\{x^1_j\}^n_{j=1}$---which would deviate from optimal transport; computing the OT map between $\{\bar{x}^0_i\}^m_{i=1}$ and an $m$-subsample of $\{x^1_j\}^n_{j=1}$---which would be greedy. With an additional sample $\{\bar{x}^1_i\}^m_{i=1}$ from $\mu_1$, we could upgrade $T_n$ to a $T_{n+m}$ by recomputing the empirical OT map between the $(n+m)$-samples. However, this would cost $\mathcal{O}\big((n+m)^3\big)$ in computer time, require new observations, and not be a natural extension of $T_n$. On the other hand, building the interpolation $\overline{T}_n$ with Karp's algorithm has a running-time complexity of $\mathcal{O}(n^3)$: the same order as for $T_n$. Then, to generalize the transport to $\{\bar{x}^0_i\}^m_{i=1}$ with $\overline{T}_n$, we must solve $m$ optimization problems, one for each $\text{prox}_{\varepsilon_0 \tilde{\varphi_n}}(\bar{x}^0_i)$. As this amount to minimizing a function which is Lipschitz with constant $\max_{1 \leq i \leq n} \norm{x^1_i} + \varepsilon^{-1}_0$ and strongly convex with constant $\varepsilon^{-1}_0$, an $\epsilon$-optimal solution can be obtained in $\mathcal{O}(\epsilon^{-1})$ steps with a subgradient descent \citep{bubeck2017convex}. Since evaluating $\partial \tilde{\varphi}_n$ at each step of the descent costs $n$ operations, computing the transport interpolation with precision $\epsilon$ of an $m$-sample has a computational complexity of order $\mathcal{O}(mn \epsilon^{-1})$. Note also that this methods is hyper-parameter free, and as such is more convenient than prior regularized approaches. In addition, the obtained map is a statistically relevant estimator: we show hereafter that the theoretical interpolation \eqref{formula} converges to the continuous OT map under mild assumptions.

\subsection{Consistency of the estimator}

We provide an extension of Proposition 3.3 in \cite{delbarrio2020centeroutward}. While the original result ensures the convergence of the interpolation $\overline{T}_n$ to $T$ in the case where $\mu_1$ is the spherical uniform law over the $d$-dimensional unit open ball, we prove that the consistency holds in more general settings.


\begin{thm}\label{convergences}

Let $\mathring{\X_0}$ and $\mathring{\X_1}$ be the respective interiors of $\X_0$ and $\X_1$, and $T$ the optimal transport map between $\mu_0$ and $\mu_1$. The following hold:

\begin{enumerate}
    \item Assume that $\mathcal{X}_0$ is convex such that $\mu_0$ has positive density on its interior. Then, for $\mu_0$-almost every $x$,

$$
    \overline{T}_n(x) \xrightarrow[n \to \infty]{a.s.} T(x).
$$

    \item Additionally assume that $T$ is continuous on $\mathring{\X_0}$, and that $\X_1$ is compact. Then, for any compact set $C$ of $\R^d$,

$$
    \sup_{x \in C}||\overline{T}_n(x) - T(x)|| \xrightarrow[n \to \infty]{a.s.} 0.
$$

In particular, provided that $\X_0$ is compact, the convergence is uniform on the support.

    \item Further assume that $\X_1$ is a strictly convex set, then

$$
    \sup_{x \in \R^d}||\overline{T}_n(x) - T(x)|| \xrightarrow[n \to \infty]{a.s.} 0.
$$
\end{enumerate}

\end{thm}
The proof falls naturally into three parts, each one dedicated to the different points of Theorem~\ref{convergences}. The first point is a consequence of Theorem 2.8 in \cite{del2019central} and Theorem 25.7 in \cite{rockafellar1970convex}, which entail that the convergence of $\{\varphi_n\}_{n \in \N}$ to $\varphi$ extends to their gradients. The proofs of the second and third points follow the guidelines of the one in \cite{delbarrio2020centeroutward}. The idea is to replace the unit ball by a compact set, and then a stricly convex set. We refer to the appendix for a complete description of this proof, as well as for all the other theoretical claims introduced in this paper.

\begin{rem}

We briefly discuss the assumptions of Theorem \ref{convergences}. Thanks to a recent work \citep{delbarrio2021central}, the convexity of $\mathcal{X}_0$ can be relaxed to having a connected support with negligible boundary.

Note that the second and third points of this theorem require a continuous optimal transport map $T$ to ensure the uniform convergence of the estimator. Caffarelli's theory \citep{ caffarelli1990localization,caffarelli1991some,caffarelli1992regularity,figalli2017monge} provides sufficient conditions for this to hold. Suppose that $\mathcal{X}_0$ and $\mathcal{X}_1$ are compact convex, and that $\mu_0$ and $\mu_1$ respectively admit $f_0$ and $f_1$ as density functions. If there exist $\Lambda\geq\lambda>0$ such that for all  $x\in\mathcal{X}_0$, $y\in\mathcal{X}_1$
\begin{align*}
    \lambda\leq f_0(x),f^{-1}_0(x),f_1(y),f^{-1}_1(y)  \leq \Lambda,
\end{align*}
then $T$ is continuous. For the non compact cases some results can be found in \cite{figalli2010partial,del2020note,corderoerausquin2019regularity}. 
\end{rem}

\section{Applications}\label{applications}

In this section, we focus on the problem of repairing and auditing the bias of a trained binary classifier. Let $(\Omega, \mathcal{A}, \P)$ be a probability space. The random vector $X : \Omega \to \R^d$ represents the observed \emph{features}, while the random variable $S: \Omega \to \{0,1\}$ encodes the observed \emph{sensitive} or \emph{protected attribute} which divides the population into a supposedly \emph{disadvantaged} class $S=0$ and a \emph{default} class $S=1$. The random variable $S$ is supposed to be non-degenerated. The two measures $\mu_0$ and $\mu_1$ are respectively defined as $\mathcal{L}(X|S=0)$ and $\mathcal{L}(X|S=1)$. The predictor is defined as $\hat{Y} := h(X,S)$, where $h : \R^d \times \{0,1\} \to \{0,1\}$ is deterministic. We consider a setting in which $\hat{Y}=1$ and $\hat{Y}=0$ respectively represent a \textit{favorable} and a \textit{disadvantageous} outcome.

\subsection{Data processing for Fair learning using Optimal Transport}
The standard way to deal with Fairness in Machine Learning is to measure it by introducing fairness measures. Among them, the disparate impact (DI) has received particular attention to determine whether a binary decision does not discriminate a minority corresponding to $S=0$, see for instance in \cite{ZVGG}. This corresponds to the notion of statistical parity introduced in \cite{dwork2012fairness}. For a classifier $h$ with values in $\{0,1\}$, set $DI(h,X,S)$ as
\[ \frac{\min(\mathbb{P}(h(X,S)=1 \mid S=0),\mathbb{P}(h(X,S)=1 \mid S=1))}{\max(\mathbb{P}(h(X,S)=1 \mid S=1),\mathbb{P}(h(X,S)=1 \mid S=0))}.\]
This criterion is close to 1 when statistical parity is ensured while the smaller the disparate, the more discrimination for the minority group. Obtaining fair predictors can be achieved by several means, one consisting in pre-processing the data by modifying the distribution of the inputs. Originally inspired by \cite{FFMSV}, this method proved in \cite{pmlr-v97-gordaliza19a}  consists in removing from the data the dependency with respect to the sensitive variable. This can be achieved by constructing two optimal transport maps, $T_0$ and $T_1$, satisfying ${T_0}_\sharp \mu_0 = \mu_B$ and ${T_1}_\sharp \mu_1 = \mu_B$, where $\mu_B$ is the Wasserstein barycenter of $\mu_0$ and $\mu_1$. The algorithm is then trained on the dataset of the modified observations following the distribution of the barycenter, which guarantees that $h \big(T_S(X),S\big) $ satisfies the statistical parity fairness criterion.

Using the estimator we propose in this work enables to compute for any new observation $(x,s)$ a prediction $h(T_{n,s}(x),s)$  with  theoretical guarantees. Note that the same framework applies when considering post-processing of outcome of estimators (or scores of classifiers) which are then pushed towards a fair representer. We refer for instance to \cite{le2020projection} for the regression case or to \cite{chiappa2020general} for the classification case.

\subsection{Consistency of Counterfactual Explanations}

A sharper approach to fairness is to \emph{explain} the discriminatory behaviour of the predictor. \cite{black2020fliptest} started laying out the foundations of auditing binary decision-rules with transport-based mappings. Their \emph{Flip Test} is an auditing technique for uncovering discrimination against protected groups in black-box classifiers. It is based on two types of objects: the \emph{Flip Sets}, which list instances whose output changes \emph{had they belonged to the other group}, and the \textit{Transparency Reports}, which rank the features that are associated with a disparate treatment across groups. Crucially, building these objects requires assessing \emph{counterfactuals}, statements on potential outcomes \emph{had a certain event occurred} \citep{lewis1973}. The machine learning community mostly focused on two divergent frameworks for computing counterfactuals: the \emph{nearest counterfactual instances} principle, which models transformations as minimal translations \citep{wachter2017counterfactual}, and Pearl's causal reasoning, which designs alternative states of things through surgeries on a causal model \citep{pearl2016causal}. While the former implicitly assumes that the covariates are independent, hence fails to provide faithful explanations, the latter requires a \emph{fully specified} causal model, which is a very strong assumption in practice. To address these shortcomings, \cite{black2020fliptest} proposed substituting causal reasoning by matching the two groups with a one-to-one mapping $T : \R^d \to \R^d$, for instance an optimal transport map. However, because the GAN approximation they use for the OT map does not come with convergence guarantees, their framework for explanability fails to be statistically consistent. We fix this issue next. More precisely, after presenting this framework, we show that natural estimators of an optimal transport map, such as the interpolation introduced in Section \ref{interpolation}, lead to consistent explanations as the sample size increases.

\subsubsection{Definitions}

In contrast to \cite{black2020fliptest}, we present the framework from a non-empirical viewpoint. The following definitions depend on the choice of the binary classifier $h$ and the mapping $T$.

\begin{defn}
For a given binary classifier $h$, and a measurable function $T : \R^d \rightarrow \R^d$, we define
    
\begin{itemize}
    \item the FlipSet as the set of individuals whose $T$-counterparts are treated unequally
    $$
        F(h,T) = \{ x \in \R^d \ |\ h(x,0) \neq h(T(x),1) \},
    $$
    \item the positive FlipSet as the set of individuals whose $T$-counterparts are disadvantaged
    $$
        F^+(h,T) = \{ x \in \R^d \ |\ h(x,0) > h(T(x),1) \},
    $$
    \item the negative FlipSet as the set of individuals whose $T$-counterparts are advantaged
    $$
        F^-(h,T) = \{ x \in \R^d \ |\ h(x,0) < h(T(x),1) \}.
    $$
\end{itemize}

When there is no ambiguity, we may omit the dependence on $T$ and $h$ in the notation.

\end{defn}

The Flip Set characterizes a set of \emph{counterfactual explanations} w.r.t. an intervention $T$. Such explanations are meant to reveal a possible bias towards $S$. The partition into a positive and a negative Flip Set sharpens the analysis by controlling whether $S$ is an advantageous attribute or not in the decision making process. As $S=0$ represents the minority, one can think of the negative partition as the occurrences of \emph{negative discrimination}, and the positive partition as the occurrences of \emph{positive discrimination}. \cite{black2020fliptest} noted that the relative sizes of the empirical positive and negative Flip Sets quantified the lack of statistical parity. Following their proof, we give a generalization of their result to the continuous case:

\begin{prop}\label{gap}
Let $h$ be a binary classifier. If $T : \mathcal{X}_0 \to \mathcal{X}_1$ satisfies $T_\sharp \mu_0 = \mu_1$, then
\begin{align*}
    \P(h(X,S)=1|S=0) &- \P(h(X,S)=1|S=1)\\ &=\\ \P(X \in F^+|S=0) &- \P(X \in F^-|S=0).\\
\end{align*}
\end{prop}

However, the interest of such sets lies in their explanatory power rather than being proxies for determining fairness scores. By analyzing the mean behaviour of $I-T$ for points in a Flip Set, one can shed light on the features that mattered the most in the decision making process. A Transparency Report indicates which coordinates change the most, in intensity and in frequency, when applying $T$ to a Flip Set. In what follows, for any $x = (x_1,\ldots,x_d)^T \in \R^d$ we define $\text{sign}(x) := (\text{sign}(x_1),\ldots,\text{sign}(x_d))^T$ the sign function on vectors.

\begin{defn} Let $\star$ be in $\{-,+\}$, $h$ be a binary classifier and $T : \mathcal{X}_0 \to \R^d$ be measurable map. Assume that $\mu_0$ and $T_\sharp \mu_0$ have finite first-order moments. The Transparency Report is defined by the mean difference vector
\begin{align*}
    &\Delta^\star_{\text{diff}}(h,T) = \E_{\mu_0}[X-T(X)|X \in F^\star(h,T)]\\
    &= \frac{1}{\mu_0(F^\star(h,T))}\int_{F^\star(h,T)} \big(x-T(x)\big) d\mu_0(x),
\end{align*}
and the mean sign vector
\begin{align*}
    &\Delta^\star_{\text{sign}}(h,T) = \E_{\mu_0}\big[\text{sign}(X-T(X))\big|X \in F^\star(h,T)]\\
    &= \frac{1}{\mu_0(F^\star(h,T))}\int_{F^\star(h,T)} \text{sign}(x-T(x)) d\mu_0(x).
\end{align*}
\end{defn}

The first vector indicates how much the points moved; the second shows whether the direction of the transportation was consistent. We upgrade the notion of Transparency Report by introducing new objects extending the Flip Test framework.

\begin{defn} Let $\star$ be in $\{-,+\}$, and $T : \mathcal{X}_0 \to \R^d$ be measurable. Assume that $\mu_0$ and $T_\sharp \mu_0$ have finite first order moments. The difference Reference Vector is defined as
$$
    \Delta^\text{ref}_{\text{diff}}(T) := \E_{\mu_0}[X-T(X)]
    = \int \big(x-T(x)\big) d\mu_0(x).
$$

and the sign Reference Vector as
\begin{align*}
    &\Delta^\text{ref}_{\text{sign}}(T) := \E_{\mu_0}[\text{sign}\big(X-T(X)\big)]\\
    &= \int \text{sign}\big(x-T(x)\big) d\mu_0(x).
\end{align*}
\end{defn}

The auditing procedure can be summarized as follows: (1) compute the Flip Sets and evaluate the lack of statistical parity by comparing their respective sizes; (2) if the Flip Sets are unbalanced, compute the Transparency Report and the Reference Vectors; (3) identify possible sources of bias by looking at the largest components of $\Delta^\star_{\text{diff}}(h,T) - \Delta^\text{ref}_{\text{diff}}(T)$ and $\Delta^\star_{\text{sign}}(h,T) - \Delta^\text{ref}_{\text{sign}}(T)$. While the original approach would have directly analyzed the largest components of the Transparency Report, the aforementioned procedure scales the uncovered variations with a reference. This benchmark is essential. It contrasts the disparity between paired instances with different outcomes to the disparity between the protected groups; thereby, pointing out the actual treatment effect of the decision rule. We give an example to illustrate how the Reference Vectors act as a sanity check for settings where the Transparency Report fails to give explanations.

\begin{ex}\label{ex}

Let $g$ be the standard gaussian measure on $\R^2$, and define $\mu_0 := (-2,-1)^T + g$ and $\mu_1 := (2,1)^T + g$, so that $\delta := \E(\mu_0-\mu_1) = -(4,2)^T$. Set $T$ as the Brenier map between $\mu_0$ and $\mu_1$, and suppose that the decision rule is $h(x_1,x_2,s) := \mathbf{1}_{\{x_2>0\}}$. In this scenario, $T$ is the uniform translation $I - \delta$, and we have
\begin{align*}
    F^-(h,T) &= \{(x_1,x_2)^T \in \R^2 \ |\ -2 < x_2 < 0 \},\\
    F^+(h,T) &= \emptyset.\\
\end{align*}
Clearly, the predictor $h$ is unfair towards $\mu_0$, since the negative FlipSet outsizes the positive one. In this case, the vector $\Delta^-_{\text{diff}}(h,T)$ is simply equal to
\begin{align*}
    \Delta^-_{\text{diff}}(h,T) &= \frac{1}{\mu_0(F^-(h,T))} \int_{F^-(h,T)} \delta d\mu_0(x)\\ &= \delta = (-4,-2)^T.
\end{align*}

A misleading analysis would state that, because $|-4|>|-2|$, the Transparency Report has uncovered a potential bias towards the first coordinate. This would be inaccurate, since the classifier only takes into account the second variable. This issue comes from the fact that in this homogeneous case, the Transparency Report only reflects how the two conditional distributions differ, and does not give any insight on the decision rule. Our benchmark approach detects such shortcomings by systematically comparing the Transparency Report to the Reference Vectors. In this setting we have $\Delta^\text{ref}_{\text{diff}}(T) = \delta$, thus $\Delta^-_{\text{diff}}(h,T) - \Delta^\text{ref}_{\text{diff}}(T) = 0$, which means that the FlipTest does not give insight on the decision making process.

\end{ex}

To sum-up, we argue that it is the deviation of $\Delta^-_{\text{diff}}(h,T)$ from $\Delta^\text{ref}_{\text{diff}}(T)$, and not $\Delta^-_{\text{diff}}(h,T)$ alone, that brings to light the possible bias of the decision rule. Note that $T_\sharp \mu_0 = \mu_1$ entails $\Delta^\text{ref}_{\text{diff}}(T) = \E[X|S=0]-\E[X|S=1]$, which does not depend on $T$. Still, we define the Reference Vector with an arbitrary $T$ because in practice we operate with an estimator that only approximates the push-forward condition. 

\subsubsection{Convergence}

The first step for implementing the Flip Test technique is computing an estimator $T_{n_0,n_1}$ of the chosen matching function $T$. In theory, the matching is not limited to an optimal transport map, but must define an intuitively justifiable notion of counterpart.
\begin{defn}\label{admissibility}
Let $T : \R^d \to \R^d$ satisfy $T_\sharp \mu_0 = \mu_1$, and $T_{n_0,n_1}$ be an estimator of $T$ built on a $n_0$-sample from $\mu_0$ and a $n_1$-sample from $\mu_1$. $T_{n_0,n_1}$ is said to be $T$-admissible if
\begin{enumerate}
    \item $T_{n_0,n_1} : \X \to \X$ is continuous on $\X_0$,
    \item $T_{n_0,n_1}(x) \xrightarrow[n_0,n_1 \to +\infty]{a.s.} T(x)$ for $\mu_0$-almost every $x$.
\end{enumerate}
\end{defn}

According to Theorem \ref{convergences}, the smooth interpolation $\overline{T}_n$ is an admissible estimator of the optimal transport map $T$ under mild assumptions.

The second step consists in building empirical versions of the Flip Sets and Transparency Reports for $h$ and $T_{n_0,n_1}$ using $m$ data points from $\mu_0$. The consistency problem at hand becomes two-fold: w.r.t. to $m$ the size of the sample, and w.r.t. to the convergence of the estimator $T_{n_0,n_1}$. Proving this consistency is crucial, as $T_{n_0,n_1}$ satisfies the push-forward condition at the limit only.

Consider a $m$-sample $\{x^0_i\}^m_{i=1}$ drawn from $\mu_0$. We define the empirical counterparts of respectively the negative Flip Set, the positive Flip Set, the mean difference vector, the mean sign vector, and the Reference Vectors for arbitrary $h$ and $T$. For any $\star \in \{-,+\}$, they are given by

\begin{align*}
    F^\star_m(h,T) &:= \{x^0_i\}^m_{i=1} \cap F^\star(h,T),\\
    \Delta^\star_{\text{diff},m}(h,T) &:= \frac{\sum^m_{i=1} \mathbf{1}_{F^\star(h,T)}(x^0_i) \big(x^0_i-T(x^0_i)\big)}{|F^\star_m(h,T)|} ,\\
    \Delta^\star_{\text{sign},m}(h,T) &:= \frac{\sum^m_{i=1} \mathbf{1}_{F^\star(h,T)}(x^0_i) \text{sign}\big(x^0_i-T(x^0_i)\big)}{|F^\star_m(h,T)|} ,\\
    \Delta^\text{ref}_{\text{diff},m}(T) &:= \frac{1}{m} \sum^m_{i=1} \big(x^0_i-T(x^0_i)\big),\\
    \Delta^\text{ref}_{\text{sign},m}(T) &:= \frac{1}{m} \sum^m_{i=1} \text{sign}\big(x^0_i-T(x^0_i)\big).
\end{align*}

Note that the first four equalities correspond to the original definitions from \cite{black2020fliptest}. The strong law of large numbers implies the convergence almost surely of each of these estimators, as precised in the following proposition.

\begin{prop}\label{firstconvergence} Let $\star \in \{-,+\}$, $h$ be a binary classifier, and $T$ a measurable function. The following convergences hold

\begin{align*}
    \frac{|F^\star_m(h,T)|}{m} &\xrightarrow[m \to +\infty]{\mu_0-a.s.} \mu_0(F^\star(h,T)),\\
    \Delta^\star_{\text{diff},m}(h,T) &\xrightarrow[m \to +\infty]{\mu_0-a.s.} \Delta^\star_{\text{diff}}(h,T),\\
    \Delta^\star_{\text{sign},m}(h,T) &\xrightarrow[m \to +\infty]{\mu_0-a.s.} \Delta^\star_{\text{sign}}(h,T),\\
    \Delta^\text{ref}_{\text{diff},m}(T) &\xrightarrow[m \to +\infty]{\mu_0-a.s.} \Delta^\text{ref}_{\text{diff}}(T),\\
    \Delta^\text{ref}_{\text{sign},m}(T) &\xrightarrow[m \to +\infty]{\mu_0-a.s.} \Delta^\text{ref}_{\text{sign}}(T).
\end{align*}

\end{prop}

In particular, theses convergences hold for an admissible estimator $T_{n_0,n_1}$. To address the further convergence w.r.t. $n_0$ and $n_1$, we first introduce a new definition.

\begin{defn}

A binary classifier $\tilde{h} : \R^d \to \{0,1\}$ is separating with respect to a measure $\nu$ on $\R^d$ if

\begin{enumerate}
    \item $H_0 := \tilde{h}^{-1}(\{0\})$ and $H_1 := \tilde{h}^{-1}(\{1\})$ are closed or open,
    \item $\nu \big(\overline{H_0} \cap \overline{H_1}\big)=0$.
\end{enumerate}

\end{defn}

We argue that except in pathological cases that are not relevant in practice, machine learning always deals with such classifiers. For example, thresholded versions of continuous functions, which account for most of the machine learning classifiers (e.g. SVM, neural networks\ldots), are separating with respect to Lebesgue continuous measures. As for a very theoretical example of non-separating classifier, one could propose the indicator of the rational numbers, which is not separating with respect to the Lebesgue measure. Working with classifiers $h$ such that $h(\cdot,1)$ is separating w.r.t. to $\mu_1$ fixes the regularity issues one might encounter when taking the limit in $h(T_{n_0,n_1}(\cdot),1)$. More precisely, it ensures that the set of discontinuity points of $h$ is $\mu_1$-negligible. As $T_\sharp \mu_0=\mu_1$ and since $T_{n_0,n_1}\rightarrow T$ $\mu_0$-almost everywhere, the following continuous mapping result holds:
\begin{prop}\label{classifconvergence}
Let $\tilde{h} : \R^d \to \{0,1\}$ be a separating classifier w.r.t. $\mu_1$, and $T_{n_0,n_1}$ a $T$-admissible estimator. Then, for $\mu_0$-almost every $x$
$$
    \tilde{h}(T_{n_0,n_1}(x)) \xrightarrow[n_0,n_1 \to +\infty]{a.s.} \tilde{h}(T(x)).
$$
\end{prop}
Next, we make a technical assumption for the convergence of the Transparency Report. Let $\{e_1,\ldots,e_d\}$ be the canonical basis of $\R^d$, and define for every $k \in \{1,\ldots,d\}$ the set $\Lambda_k(T) := \{x \in \R^d \ |\ \langle x - T(x), e_k \rangle = 0 \}$.

\begin{hyp}\label{axis}
For every $k \in \{1,\ldots,d\}$, $\mu_0\big(\Lambda_k(T)\big) = 0$.
\end{hyp}

Any Lebesgue continuous measure satisfies Assumption \ref{axis}. This is crucial for the convergence of the mean sign vector, as it ensures that the points of discontinuity of $x \mapsto \text{sign}\big(x-T(x)\big)$ are negligible. We now turn to our main consistency result.

\begin{thm}\label{ftconsistency}
Let $\star \in \{-,+\}$, $h$ be  a binary classifier such that $h(\cdot,1)$ is separating w.r.t. $\mu_1$, and $T_{n_0,n_1}$ a $T$-admissible estimator. The following convergences hold
\begin{align*}
    \mu_0(F^\star(h,T_{n_0,n_1})) &\xrightarrow[n_0,n_1 \to +\infty]{a.s.} \mu_0(F^\star(h,T)),\\
    \Delta^\star_{\text{diff}}(h,T_{n_0,n_1}) &\xrightarrow[n_0,n_1 \to +\infty]{a.s.} \Delta^\star_{\text{diff}}(h,T),\\
    \Delta^\text{ref}_{\text{diff}}(T_{n_0,n_1}) &\xrightarrow[n_0,n_1 \to +\infty]{a.s.} \Delta^\text{ref}_{\text{diff}}(T).
\end{align*}
If Assumption \ref{axis} holds, then additionally 
\begin{align*}
    \Delta^\star_{\text{sign}}(h,T_{n_0,n_1}) &\xrightarrow[n_0,n_1 \to +\infty]{a.s.} \Delta^\star_{\text{sign}}(h,T),\\
    \Delta^\text{ref}_{\text{sign}}(T_{n_0,n_1}) &\xrightarrow[n_0,n_1 \to +\infty]{a.s.} \Delta^\text{ref}_{\text{sign}}(T).
\end{align*}
\end{thm}

As $h$ is binary, the probability of the negative Flip Set can be written as $\mu_0(F^-(h,T_{n_0,n_1})) = \int [1-h(x,0)]h(T_{n_0,n_1}(x),1) d\mu_0(x).$ Note that the integrated function $[1-h(\cdot,0)]h(T_{n_0,n_1}(\cdot),1)$ is dominated by the constant $1$. Then, it follows from Proposition \ref{classifconvergence} that this sequence of functions converges $\mu_0$-almost everywhere to $[1-h(\cdot,0)]h(T(\cdot),1)$ when $n_0,n_1 \to +\infty$. By the dominated convergence theorem, we conclude that $\mu_0(F^-(h,T_{n_0,n_1})) \xrightarrow[n_0,n_1 \to +\infty]{}\mu_0(F^-(h,T))$. The same argument holds for the positive Flip Sets. The proofs of the other convergences follow the same reasoning, using Proposition \ref{classifconvergence} and Assumption \ref{axis} to apply the dominated convergence theorem.\\

As aforementioned, the assumptions of Theorem \ref{ftconsistency} are not significantly restrictive in practice. Thus, the Flip Test framework is tailored for implementations.

\section{Conclusion}

We addressed the problem of constructing a statistically approximation of the continuous optimal transport map. We argued that this has strong consequences for machine learning applications based on OT, as it renders possible to generalize discrete optimal transport on new observations while preserving its key properties. We illustrated that using the proposed extension ensures the statistical consistency of OT-based frameworks, and as such derived the first consistency analysis for observation-based counterfactual explanations.


\bibliographystyle{plainnat}
\bibliography{references}

\appendix

\section{Proofs of Section \ref{learning}}

\subsection{Intermediary result}

We first introduce a proposition adapted from \cite{delbarrio2020centeroutward} to suit our setting. In what follows, we denote by $N_C(x) := \{y \in \R^d |\ \forall x' \in C,\ \langle y,x'-x\rangle \leq 0\}$ the \emph{normal cone} at $x$ of the convex set $C$. 

\begin{prop}\label{limit}
Suppose that $\mathcal{X}_1$ is a compact convex set. Let $x_n = \lambda_n u_n \in \R^d$ where $0 < \lambda_n \to +\infty$ and $u_n \in \partial \mathcal{X}_1 \to u$ as $n \to +\infty$. Note that by compactness of the boundary, necessarily $u \in \partial \mathcal{X}_1$. If $(T(x_n))_{n \in \N}$ has a limit $v$ (taking a subsequence if necessary), then $
v \in \partial \mathcal{X}_1,
$
and 
$
    u \in N_{\mathcal{X}_1}(v) \neq \{0\}.    
$
\end{prop}
\begin{proof}
Taking subsequences if necessary, we can assume that $T(x_n) \to v$ for some $v \in \mathcal{X}_1$. The monotonicity of $T$ implies that for any $x \in \R^d$,
$
\langle x_n-x , T(x_n)-T(x) \rangle \geq 0.
$
In particular, for any $w \in T(\R^d)$,
$
\langle x_n-T^{-1}(w) , T(x_n)-w \rangle \geq 0.
$
This can be written as
$
\langle u_n-\frac{1}{\lambda_n} T^{-1}(w) , T(x_n)-w \rangle \geq 0.
$
Taking the limit leads to
$
\langle u , v-w \rangle \geq 0.
$
Define $H := \{w \in \R^d \ |\ \langle u , w-v \rangle \leq 0\}$ which is a closed half-space. As $T$ pushes $\mu_0$ towards $\mu_1$, $T(\R^d)$ contains a dense subset of $\mathcal{X}_1$. Since $H$ is closed, this implies that $\mathcal{X}_1 \subset H$ and $v\in \X_1\cap H$. Consequently, $H$ is a supporting hyperplane of $\X_1$ and $v\in \partial\X_1$. Now, write the inclusion $\X_1 \subset H$ as $\forall{w \in \mathcal{X}_1},\ \langle u , w-v \rangle \leq 0$. Denote by $N_{\mathcal{X}_1}(x)$ the normal cone of $\mathcal{X}_1$ at an arbitrary point $x \in \R^d$. Conclude by noting that the above inequality reads $u \in N_{\mathcal{X}_1}(v)$. The cone does not narrow down to $\{0\}$ as $v$ does not belong to $\mathring{\mathcal{X}_1}$.
\end{proof}

\subsection{Proof of Theorem \ref{convergences}}
We now turn to the proof of the main theorem.

\begin{proof}

Recall that $\mu_0$ and $\mu_1$ are probability measures on $\R^d$ with respective supports $\mathcal{X}_0$ and $\mathcal{X}_1$. We denote their interiors by $\mathring{\mathcal{X}_0}$ and $\mathring{\mathcal{X}_1}$, and their boundaries by $\partial \mathcal{X}_0$ and $\partial \mathcal{X}_1$. We assume the measures to be absolutely continuous with respect to the Lebesgue measure. Recall that there exists an unique map $T$ such that $T_\sharp \mu_0 = \mu_1$ and $T = \nabla \varphi$ $\mu_0$-almost everywhere for some convex function $\varphi$ called a \emph{potential}. We denote by $\text{dom}(\nabla \varphi)$ the set of differentiable points of $\varphi$ which satisfies $\mu_0(\text{dom}(\nabla \varphi))=1$, according to Theorem 25.5 in \cite{rockafellar1970convex}.

Conversely, there also exists a convex function $\psi$ such $S$, the Brenier's map from $\mu_1$ to $\mu_0$, can be written as $S := \nabla \psi$ $\mu_1$-almost everywhere. In addition, $S$ can be related to $T$ through the potential functions. Concretely, $\psi$ coincides with the convex conjugate 
$
   \varphi^*(y) = \sup_{x\in \R^d} \big\{ \langle x,y\rangle - \varphi(x) \big\}
$
of $\varphi$. We can then fix this function for $u \in \R^d \setminus \mathring{\mathcal{X}_1}$ using the lower semi-continuous extension on the support. This defines a specific $\varphi$ (hence a specific solution $T$) as

\begin{equation}\label{extension}
        \varphi(x) := \sup_{u \in \R^d} \big\{ \langle x,u \rangle - \varphi^*(u) \big\} = \sup_{u \in \mathcal{X}_1} \big\{ \langle x,u \rangle - \varphi^*(u) \big\}.
\end{equation}

Let $\{x^0_i\}^n_{i=1}$ and $\{x^1_i\}^n_{i=1}$ be $n$-samples drawn from respectively $\mu_0$ and $\mu_1$, defining empirical measures $\mu^n_0$ and $\mu^n_1$. Without loss of generality, assume that the samples are ordered such that $T_n : x^0_i \mapsto x^1_i$ is the unique solution to the corresponding discrete Monge problem. Consider the interpolation $\overline{T}_n$. We pay attention to the properties it satisfies:

\begin{enumerate}
    \item $\overline{T}_n = \nabla \varphi_n$ where $\varphi_n$ is continuously differentiable,
    \item $\overline{T}_n$ is cyclically monotone,
    \item for all ${i \in \{1,\ldots,n\}},\ \overline{T}_n(x^0_i) = x^1_i = T_n(x^0_i)$,
    \item for all ${x \in \R^d},\ \overline{T}_n(x) \in \text{conv}\big( \{x^1_1,\ldots,x^1_n\} \big)$.
\end{enumerate}
Following the decomposition of Theorem~\ref{convergences} the proof will be divided into three steps.
\paragraph{Step 1: Point-wise convergence.}

Assume that the support $\mathcal{X}_0$ is a convex set. Recall that $\overline{T}_n= \nabla \varphi_n$ everywhere and $T = \nabla \varphi$ $\mu_0$-almost everywhere. We prove the point-wise convergence of $\{\overline{T}_n\}_{n \in \N}$ to $T$ in two steps: first, we show the point-wise convergence of $\{\varphi_n\}_{n \in \N}$ to $\varphi$; second, we do the same for $\{\nabla \varphi_n\}_{n \in \N}$ to $\nabla \varphi$.

Theorem 5.19 in \cite{villani2008optimal} implies that 
$$\gamma^n = (I \times T_n)_\sharp \mu^n_0 \xrightarrow[n \to +\infty]{w} \gamma = (I \times T)_\sharp \mu_0,$$
where $w$ denotes the weak convergence of probability measures. 
It follows from Theorem 2.8 in \cite{del2019central} that for all $x \in \mathring{\mathcal{X}_0}$ the limit $\lim_{n \to +\infty} \varphi_n(x) = \varphi(x)$ holds after centering.

Theorem 25.5 in \cite{rockafellar1970convex} states that for all $x \in \text{dom}(\nabla \varphi)$ there exist an open convex subset $C$ such that $x \in C \subset \text{dom}(\nabla \varphi)$. Take an arbitrary $x \in \mathring{\mathcal{X}_0} \cap \text{dom}(\nabla \varphi)$ and consider such a subset $C$ containing $x$. Since $\varphi$ is finite and differentiable in $C$, we can apply Theorem 25.7 in \cite{rockafellar1970convex} to conclude that $\nabla \varphi(x) = \lim_{n \to +\infty} \nabla \varphi_n(x)$. To sum-up, the desired equality holds in $\mathring{\mathcal{X}_0} \cap \text{dom}(\nabla \varphi)$, in consequence $\mu_0$-almost surely (recall that the border of a convex set is Lebesgue negligible).

\paragraph{Step 2: Uniform convergence on the compact sets.}

Further assume that $\nabla \varphi$ is continuous on $\mathring{\mathcal{X}_0}$, and that the support $\mathcal{X}_1$ is a compact set. Set $K_1 = \sup_{x \in \mathcal{X}_1} ||x||$. This implies that for any ${x \in \R^d}$,
$||\overline{T}_n(x)|| \leq \max_{1 \leq i \leq n} ||x^1_i|| \leq K_1.$ Then  $||\nabla\varphi_n(x)||\leq K_1$ for all $n\in \N$ and $x\in\R^d$. In consequence the sequence $\{\varphi_n\}_{n\in\N}$ is equicountinous with the topology of convergence on the compact sets. Arzela-Ascoli's theorem applied on the compact sets of $\R^d$ implies that the sequence is s relatively compact in the topology induced by the uniform norm on the compact sets. Let $\rho$ be any cumulative point of $\{\varphi_n\}_{n\in\N}$. Then, there exists a sub-sequence of $\{\varphi_n\}_{n\in\N}$ converging to $\rho$.  Abusing notation, we keep denoting the sub-sequence by $\{\varphi_n\}_{n\in\N}$. The previous step implies that $\varphi=\rho$ and $\nabla \varphi =\nabla \rho$ on $\X_0$. Next, we show that this equality holds on $\R^d \setminus \X_0$.

The continuity of the transport map implies that $\mathring{\mathcal{X}_0}\subset\text{dom}(\nabla\varphi)$.
Hence, by convexity, for every $z \in \R^d$ and $u = \nabla \varphi(x) = \nabla \rho(x) \in \nabla \varphi (\mathring{\X_0})$, 
\begin{equation}\label{eq:ineq_fench}
\rho(z)\geq \rho(x)+ \langle u, z-x\rangle=\langle u,x \rangle - \varphi^*(u),
\end{equation}
where the equality comes from the equality case of the Fenchel-Young theorem. As $\mu_0(\mathring{\X_0})=1$, the push-forward condition $\nabla \varphi_\sharp \mu_0 = \mu_1$ implies that $\mu_1(\nabla \varphi(\mathring{X_0}))=1$ and consequently $\nabla \varphi(\mathring{X_0})$ is dense in $\X_1$. It follows that
\begin{align}\label{eq:inequality_main1}
   \rho(z) \geq \sup_{u \in  \nabla \varphi(\mathring{\X_0})} \big\{ \langle u,z\rangle - \varphi^*(u) \big\} = \sup_{u \in \mathcal{X}_1} \big\{ \langle x,u \rangle - \varphi^*(u) \big\} = \varphi(z)\ \text{for every $z \in \R^d$}. 
\end{align}
To get the upper bound, set $z \in \R^d$ and $u_n = \nabla \varphi_n(z) = \overline{T}_n(z)$. Since $\overline{T}_n(z) \in \text{conv}\big( \{x^1_1,\ldots,x^1_n\} \big)$, then $u_n \in \mathcal{X}_1$. Fenchel-Young equality once again implies that
$
    \langle x,u_n \rangle = \varphi_n(x) + \varphi^*_n(u_n).
$ This gives that 

$$
    \varphi_n(x) \leq \sup_{u \in \mathring{\X_1}} \big\{ \langle u,x \rangle - \varphi^*_n(u) \big\} = \tilde{\varphi}_n(x),
$$
where $\tilde{\varphi}_n$ is the Legendre transform of
$$
    \tilde{\varphi}^*_n : u \mapsto \begin{cases*}
      \varphi^*_n(u) & if $u \in \mathring{\X_1}$,\\
      +\infty        & otherwise.
    \end{cases*}
$$
Since $\nabla \varphi^*$ is the Brenier map from $\mu_1$ to $\mu_0$, then Theorem 2.8 in \cite{del2019central} implies that $\lim_{n \to +\infty} \varphi^*_n(u) = \varphi^*(u) = \lim_{n \to +\infty} \tilde{\varphi}^*_n(u)$ for every $u \in \mathring{\X_1}$. Outside $\mathring{\X_1}$ we have $\tilde{\varphi}^*_n(u) = +\infty = \varphi^*(u)$ by definition. Hence, the sequence $\{\tilde{\varphi}^*_n\}_{n\in\N}$ converges point-wise to $\varphi^*$ over $\R^d$. According to Theorem~7.17 in together with Theorem~11.34 in \cite{rockafellar2009variational} the same convergence holds for their conjugates. This means that for any $x \in \R^d$ we have $\lim_{n \to +\infty} \tilde{\varphi}_n(x) = \varphi(x)$. This leads to $\rho(x) \leq \varphi(x)$ for every $x \in \R^d$, hence $\rho = \varphi$. We conclude, using Theorem~25.7 in \cite{rockafellar1970convex}, that $\overline{T}_n = \nabla \varphi_n$ converges uniformly to $T = \nabla \varphi$ over compact sets of $\R^d$, in particular over $\mathcal{X}_0$ if it is compact.

\paragraph{Step 3: Uniform convergence on $\R^d$.}

Further assume that the support $\mathcal{X}_1$ is a strictly convex set. To prove the result it suffices to show that, for every $w \in \R^d$,

$$
    \sup_{x \in \R^d} | \langle \overline{T}_n(x)-T(x) , w \rangle | \xrightarrow[n \to +\infty]{} 0.
$$

Let's assume that on the contrary, there exist $\varepsilon >0$, $w \neq 0$, and $\{x_n\}_{n\in\N} \subset \R^d$ such that

\begin{equation}\label{contradicts}
    | \langle \overline{T}_n(x_n)-T(x_n) , w \rangle | > \varepsilon
\end{equation}

for all $n$. Necessarily, the sequence $\{x_n\}_{n\in\N}$ is unbounded. If not, we could extract a convergent subsequence so that, by using the point-wise convergence and the continuity of the transport functions, the left-term of \eqref{contradicts} would tend to zero. Taking subsequences if necessary, we can assume that $x_n = \lambda_n u_n$ where $\lim_{n \to +\infty} u_n = u$ where $u_n, u \in \partial \mathcal{X}_1$ and $0 < \lambda_n \to +\infty$. By compactness of $\mathcal{X}_1$ and Proposition \ref{limit}, $\overline{T}_n(x_n) \to z \in \mathcal{X}_1$ and $T(x_n) \to y \in \partial \mathcal{X}_1$. Let $\tau > 0$ so that by monotonicity

$$
\langle \overline{T}_n(x_n)-\overline{T}_n(\tau u_n) , (\lambda_n - \tau)u_n \rangle \geq 0.
$$

For $n$ large enough so that $\lambda_n > \tau$ we have

$$
    \langle \overline{T}_n(x_n)-T(\tau u_n) ,u_n \rangle + \langle T(\tau u_n)-\overline{T}_n(\tau u_n) ,u_n \rangle\geq 0.
$$

The second term tends to zero, leading to

$$
    \langle z-T(\tau u) ,u \rangle \geq 0.
$$

As this holds for any $\tau > 0$, we can take $\tau_n = \lambda_n \to +\infty$ to get
\begin{align}\label{eq:togetherwith}
    \langle z-y ,u \rangle \geq 0.
\end{align}

According to Proposition \ref{limit}, $u \in N_{\mathcal{X}_1}(y)$ with $u \neq 0$. In particular, as $z \in \mathcal{X}_1$, we have that $\langle u , z - y \rangle \leq 0$, which implies that $\langle u , z - y \rangle = 0$.
This means that $u \perp z - y$ and $ u\in N_{\mathcal{X}_1}(y)$. Hence, $z-y $ belongs to the tangent plane of $\partial \mathcal{X}_1$ at $y$ while $z \in \mathcal{X}_1$. Besides, $\mathcal{X}_1$ is strictly convex, implying that $z = y$. This contradicts at the limit with \eqref{contradicts}.

\end{proof}

\section{Proof of Section \ref{applications}}

Proof of Proposition \ref{gap}.

\begin{proof}

Note that

$$
    F^-(h,T) = \{h(x,0)=0 \text{ and } h(T(x),1)=1\}
             = \{h(T(x),1)=1\} - \{h(x,0)=1 \text{ and } h(T(x),1)=1\}.
$$
    
Similarly

$$
    F^+(h,T) = \{h(x,0)=1 \text{ and } h(T(x),1)=0\}
             = \{h(x,0)=1\} - \{h(x,0)=1 \text{ and } h(T(x),1)=1\}.
$$

Taking the measures we get
$$
\mu_0(F^-)-\mu_0(F^+)
= \mu_0(\{ x \in \R^d \ |\ h(T(x),1)=1 \}) - \mu_0(\{ x \in \R^d \ |\ h(x,0)=1 \}).
$$

Using the fact that $T_\sharp \mu_0 = \mu_1$ we have

$$
    \mu_0(\{ x \in \R^d \ |\ h(T(x),1)=1 \}) = \mu_1(\{ x \in \R^d \ |\ h(x,1)=1 \}).
$$

This leads to

$$
    \mu_0(F^-)-\mu_0(F^+) =  \mu_1(\{ x \in \R^d \ |\ h(x,1)=1 \}) - \mu_0(\{ x \in \R^d \ |\ h(x,0)=1 \}).
$$

Which concludes the proof.

\end{proof}

Proof of Proposition \ref{firstconvergence}.

\begin{proof}

Let $\star \in \{-,+\}$. The empirical probability of the Flip Set is

$$
    \frac{|F^\star_m(h,T)|}{m} = \frac{1}{m} \sum^m_{i=1} \mathbf{1}_{F^\star(h,T)}(x^0_i).
$$

By the strong law of large numbers,

$$
    \frac{1}{m} \sum^m_{i=1} \mathbf{1}_{F^\star(h,T)}(x^0_i) \xrightarrow[m \to +\infty]{\mu_0-a.s.}\ \E_{\mu_0}[\mathbf{1}_{F^\star(h,T)}(X)] = \mu_0(F^\star(h,T)).
$$

This concludes the first part of the proof. We now turn to the Transparency Report, and show the convergence of the mean difference vector, as the proof is equivalent for the mean sign vector. The empirical estimator can be written as

$$
    \Delta^\star_{\text{diff},m}(h,T) = \frac{m}{|F^\star_m(h,T)|} \times m^{-1} \sum^m_{i=1} \mathbf{1}_{F^\star(h,T)}(x^0_i) \big(x^0_i-T(x^0_i)\big).
$$

Then, by the strong law of large numbers we have

$$
    \frac{m}{|F^\star_m(h,T)|} \times m^{-1} \sum^m_{i=1} \mathbf{1}_{F^\star(h,T)}(x^0_i) \big(x^0_i-T(x^0_i)\big) \xrightarrow[m \to +\infty]{\mu_0-a.s.} \frac{1}{\mu_0(F^\star(h,T))}\int_{F^\star(h,T)} \big(x-T(x)\big) d\mu_0(x),
$$

where by definition

$$
    \frac{1}{\mu_0(F^\star(h,T))} \int_{F^\star(h,T)} \big(x-T(x)\big) d\mu_0(x) =  \Delta^\star_{\text{diff}}(h,T).
$$

The proof for the Reference Vectors is identical, even simpler as $h$ is not involved.

\end{proof}

Proof of Proposition \ref{classifconvergence}.

\begin{proof}

Throughout this proof, we work with a given realization $T_{n_0,n_1} := T^{(\omega)}_{n_0,n_1}$ of the random estimator for an unimportant arbitrary $\omega \in \Omega$. Without loss of generality, consider that $H_0$ is open and $H_1$ is closed. Recall that by $T$-admissibility, the sequence $T_{n_0,n_1}(x)$ converges for $\mu_0$-almost every $x$.  We aim at showing that for $\mu_0$-almost every $x$ 

$$
    \tilde{h}(T_{n_0,n_1}(x)) \xrightarrow[n_0,n_1 \to +\infty]{} \tilde{h}(T(x)).
$$

For any $x\in \X_0$ there are only two different cases.

\paragraph{Case 1:} For any $n_0$ and $n_1$ large enough, $T_{n_0,n_1}(x) \in H_1$. Then at the limit, $T(x) \in H_1$, meaning that the expected convergence holds for this $x$.

\paragraph{Case 2:} For any $n_0$ and $n_1$ large enough, $T_{n_0,n_1}(x) \in H_0$. Then at the limit, either $T(x) \in H_0$ or $T(x) \in H_1$. If $T(x) \in H_0$, the expected convergence holds for this $x$. If $T(x) \in H_1$, necessarily $T(x) \in \overline{H_0} \cap H_1$. As $\mu_1(\overline{H_0} \cap H_1) = 0$ and $\mu_1 = \mu_0 \circ T^{-1}$, this only occurs for $x$ in a $\mu_0$-negligible set.

Any other cases would contradict with the convergence of $T_{n_0,n_1}(x)$. Consequently, the expected convergence holds $\mu_0$-almost everywhere.

\end{proof}

Proof of Theorem \ref{ftconsistency}.

\begin{proof}

In this proof as well we work with a given realization $T_{n_0,n_1} := T^{(\omega)}_{n_0,n_1}$ of the random estimator for an unimportant arbitrary $\omega \in \Omega$. Let's show the result for $\star = -$ as the proof is equivalent for $\star = +$. Because $h$ is binary, the probability of the Flip Set can be written as

$$
    \mu_0(F^-(h,T_{n_0,n_1})) = \E_{\mu_0}[\mathbf{1}_{\{F^-(h,T_{n_0,n_1})\}}]
    = \int [1-h(x,0)]h(T_{n_0,n_1}(x),1) d\mu_0(x).
$$

Note that the integrated function $[1-h(\cdot,0)]h(T_{n_0,n_1}(\cdot),1)$ is dominated by the constant $1$. Then, it follows from Proposition \ref{classifconvergence} that this sequence of functions converges $\mu_0$-almost everywhere to $[1-h(\cdot,0)]h(T(\cdot),1)$ when $n_0,n_1 \to +\infty$. By the dominated convergence theorem, we conclude that

$$
    \mu_0(F^-(h,T_{n_0,n_1})) \xrightarrow[n_0,n_1 \to +\infty]{} \mu_0(F^-(h,T)).
$$

We now turn to the mean difference vector,

$$
    \Delta^-_{\text{diff}}(h,T_{n_0,n_1}) = \frac{1}{\mu_0(F^-(h,T_{n_0,n_1}))}\int_{F^-(h,T_{n_0,n_1})} \big(x-T_{n_0,n_1}(x)\big) d\mu_0(x).
$$

We already proved that the left fraction converges to $\mu_0(F^-(h,T))^{-1}$. To deal with the integral, we exploit once again the fact that $h$ is binary to write

$$
    \int_{F^-(h,T_{n_0,n_1})} \big(x-T_{n_0,n_1}(x)\big) d\mu_0(x) = \int [1-h(x,0)] h(T_{n_0,n_1}(x),1) \big(x-T_{n_0,n_1}(x)\big) d\mu_0(x).
$$

Note that the sequence of functions $x \mapsto x-T_{n_0,n_1}(x)$ converges $\mu_0$-almost everywhere to $x \mapsto x-T(x)$. This is where \ref{classifconvergence} comes into play to ensure the convergence $\mu_0$-everywhere of the integrated function. This enables to apply the dominated convergence theorem to conclude that

$$
     \Delta^-_{\text{diff}}(h,T_{n_0,n_1})  \xrightarrow[n_0,n_1 \to +\infty]{} \Delta^-_{\text{diff}}(h,T).
$$

We finally address the case of the mean sign vector:

$$
    \Delta^-_{\text{sign}}(h,T_{n_0,n_1}) = \int [1-h(x,0)] h(T_{n_0,n_1}(x),1) \text{sign}\big(x-T_{n_0,n_1}(x)\big) d\mu_0(x).
$$

The approach is the same as for the mean difference vector. The only crucial distinction to handle is the convergence of the sequence $x \mapsto \text{sign}\big(x-T_{n_0,n_1}(x)\big)$ to $x \mapsto \text{sign}\big(x-T(x)\big)$, which is not trivial as the sign function is discontinuous wherever a coordinate of its argument equals zero. We follow a similar reasoning as for the proof of Proposition \ref{classifconvergence} to show the convergence $\mu_0$-almost everywhere. The only pathological case happens when $x-T(x)$ ends up on a canonical axis, that is to say when $x \in \Lambda_k(T)$. If Assumption \ref{axis} holds, this occurs only for $x$ in a $\mu_0$-negligible set. Consequently, for $\mu_0$-almost every $x$
 
$$
    \text{sign}\big(x-T_{n_0,n_1}(x)\big) \xrightarrow[n_0,n_1 \to +\infty]{} \text{sign}\big(x-T(x)\big).
$$

To conclude, we apply Proposition \ref{classifconvergence} along with the dominated convergence theorem to obtain

$$
    \Delta^-_{\text{sign}}(h,T_{n_0,n_1}) \xrightarrow[n_0,n_1 \to +\infty]{} \Delta^-_{\text{sign}}(h,T).
$$

The proof for the Reference Vectors is identical, even simpler as $h$ is not involved.
\end{proof}

\end{document}